\documentclass[11pt,a4paper,thmdefs]{amsart}
\usepackage{amsfonts}
\usepackage{amssymb,amsmath}
\usepackage{amscd}
\usepackage{graphicx}

\newtheorem{theorem}{Theorem}[section]
\newtheorem{lemma}{Lemma}[section]

\newtheorem{corollary}{Corollary}[section]
\newtheorem{definition}{Definition}[section]
\newtheorem{example}{Example}[section]

\numberwithin{equation}{section} \textwidth12.5cm
\input{tcilatex}

\begin{document}
\title[]{How can we construct a $k-$slant curve from a given spherical curve?%
}
\author[]{\c{C}etin CAMCI}
\address{Department of Mathematics \\
Onsekiz Mart University \\
17020 \c{C}anakkale, Turkey}
\email{ccamci@comu.edu.tr}
\date{}

\begin{abstract}
Helices and constant procession curves are special examples of the slant
curves. \ However, there is no an example of the k-slant curve $(2\leq k)$
in Euclidean 3-space. Furthermore, the position vector of a $k-$slant curve $%
(2\leq k)$ has not been known until this time.

The aim of the paper is to give a method of constructing$\ k$-slant curves
via any plane curve. This idea provide that spherical both $k$-slant curves
and $N_{k}$-constant procession curves can be derived from the circle, for $%
k\in 
\mathbb{N}
$. In addition, we give a new proof of the spherical curve characterization
and define a curve in the sphere called spherical prime curve. Eventually,
we give an application of \ a $k$-slant curve in the magnetism theory and
illustrated examples of the curves..
\end{abstract}

\maketitle


\section{\textbf{Introduction}}

Curves are geometric set of points of loci in spaces. In the differential
geometry studies, special curves such as geodesics, circles, circular
helices, general helices, slant helices, $C-$slant curves, glad helices etc.
have been studied for a long time in different spaces. We can see helical
structures in nature, physics, kinematic motion, design of architectural
building and the structure of the DNA \ which is double helix. A curve
called as general helix (or constant slope curve) if its tangent vector
field makes a constant angle with a fixed straight line. In 1802, Lancert (%
\cite{L}) proposed a classical characterization of general helix and then
this open problem was solved by Saint-Venant (\cite{SV}) with a condition
that a curve is a general helix iff its the ratio of curvature to torsion is
constant. In addition, if its curvature functions are non-zero constants
then the curve is called as a circular helix. Furthermore, straight lines
and circles are degenerate-helices.

Generalized of the genaral helix is a slant curve. A slant helix has been
defined by the property that its principal normal vector field makes a
constant angle with a fixed direction. They proved that $\alpha $ is a slant
helix if and only if the geodesic curvature of spherical image of principal
normal indicatrix of a space curve $\gamma $ 
\begin{equation*}
\sigma =\frac{\kappa ^{2}}{\left( \kappa ^{2}+\tau ^{2}\right) ^{\frac{3}{2}}%
}\left( \frac{\tau }{\kappa }\right) ^{\prime }
\end{equation*}%
is a function (\cite{IT}). Kula and Yayl\i\ denote some properties of the
slant curve (\cite{KY}) and Camci et al. studied spherical slant curve in
Euclidean $3-$space (\cite{C}). We can give a special slant helix as a
Salkowski curve whose curvature function $\kappa $ is constant (\cite{SAL}).
Recently, Ali (\cite{Ali}) has defined a curve called $k-$slant helix and
has given various characterizations. Moreover, Takahashi and Takeuchi (\cite%
{TAK}) has defined new special curves in Euclidean $3-$space which are
generalization of the notion of helical curves called clad helices ($2-$%
slant) and g-clad ($3-$slant) helices. Scofield (\cite{Sco}) have given a
curve of constant precession by the property that the curve is traversed
with unit speed, its centrode has a constant angle with a fixed axis and
revolves constant speed. Then, Uzunoglu et.al(\cite{UZ}) gave a new approach
of a curve of constant precession with alternative moving frame.

Actually, definition of slant curve was given by Blum ( \cite{RB}) in 1966
with the same words.\ Blum studied a curve whose curvatures are equal to $%
\kappa (s)=w\cos \left( as+b\right) $ and $\tau (s)=w\sin \left( as+b\right)
.$ However, Mannheim (\cite{M}) investigated a class of curves which
curvature functions has following relation $\kappa ^{2}+\tau ^{2}=w^{2}$ ($%
w= $constant) in 1878. Then, Blum called that curves of constant precession
are Mannheim curves with a theorem. The definition of slant curves were
firstly defined by Blum in Euclidean $3-$space(\cite{RB}) but was coined by
Izumiya and Tkeuchi in 2004 (\cite{IT}).

In the present paper, first of all, we give a method of constructing$\ k$%
-slant curves via any a plane curve. This idea provide that spherical both $%
k $-slant curves and $N_{k}$-constant procession curves can be derived from
the circle, for $k\in 
\mathbb{N}
$. In addition, we give a new proof of the spherical curve characterization
and define a curve in the sphere called spherical prime curve. The most
important point is that spherical helices oscillate in the equator of the
sphere.

\section{\textbf{Basic concepts and notions}}

In Euclidean $3-$spaces, let $\gamma $ be a unit-speed curve with coordinate
neighborhood $(I,\gamma )$ and $\left\{ T,N,B,\kappa ,\ \tau \right\} $ be
Serret-Frenet apparatus of the curves. Derivation of the Serret-Frenet
vectors is given by 
\begin{equation*}
\left( 
\begin{array}{c}
T%
{\acute{}}
\\ 
N^{%
{\acute{}}%
} \\ 
B^{%
{\acute{}}%
}%
\end{array}%
\right) =\left( 
\begin{array}{ccc}
0 & \kappa & 0 \\ 
-\kappa & 0 & \tau \\ 
0 & -\tau & 0%
\end{array}%
\right) \left( 
\begin{array}{c}
T \\ 
N \\ 
B%
\end{array}%
\right)
\end{equation*}%
(\cite{F}, \cite{S}) . Centrode of the curve $C$ is define as 
\begin{equation*}
W(s)=\tau (s)T(s)+\kappa (s)B(s)
\end{equation*}%
(\cite{ST}). If the curve is a spherical curve, then it is well known that $%
\gamma (s)$ is perpendicular to $T=\gamma ^{\prime }(s)$, for all $s\in I$.
So we have an orthonormal frame $\left\{ \gamma (s)\text{, }T(s)=\gamma
^{\prime }(s)\text{, }Y(s)=\gamma (s)\times T(s)\right\} $ along $\gamma $.
This frame of the curve is called the Sabban frame along $\gamma $ (\cite%
{ITB}). Serret-Frenet formula of the spherical curve is given by 
\begin{equation*}
\left( 
\begin{array}{c}
\gamma ^{\prime } \\ 
T^{\prime } \\ 
Y^{\prime }%
\end{array}%
\right) =\left( 
\begin{array}{ccc}
0 & 1 & 0 \\ 
-1 & 0 & \kappa _{g} \\ 
0 & -\kappa _{g} & 0%
\end{array}%
\right) \left( 
\begin{array}{c}
\gamma \\ 
T \\ 
Y%
\end{array}%
\right)
\end{equation*}%
where $\kappa _{g}(s)=\det \left( \gamma ,T,T^{\prime }\right) $ is called
geodesic curvature along the curve in 2-sphere (\cite{ITB}).

In recent years, Ali has defined a unit vector as follow 
\begin{equation}
\psi _{k+1}(s)=\frac{\psi _{k}^{\prime }}{\left\Vert \psi _{k}^{\prime
}\right\Vert }
\end{equation}%
where $\psi _{0}(s)=\gamma (s)$, $\psi _{1}(s)=T(s)$, $\psi _{2}(s)=N(s)$(%
\cite{Ali}). So it can be defined a regular curve $\gamma _{k}$ as 
\begin{equation}
\gamma _{k}(s)=\overset{s}{\underset{0}{\dint }}\psi _{k+1}(u)du
\end{equation}%
(\cite{CY}). Let $\left\{ T_{k},N_{k},B_{k},\kappa _{k},\ \tau _{k}\right\} $
be Serret-Frenet apparatus of the curves $\gamma _{k}$. So it can be easily
obtain\ 
\begin{equation*}
T_{k}=\psi _{k+1},N_{k}=\frac{\psi _{k+1}^{\prime }}{\left\Vert \psi
_{k+1}^{\prime }\right\Vert }=\psi _{k+2}=T_{k+1},B_{k}=T_{k}\times N_{k}
\end{equation*}%
(\cite{Ali}). Serret-Frenet formula of the curves $\gamma _{k}$ is given by 
\begin{equation*}
\left( 
\begin{array}{c}
T_{k}%
{\acute{}}
\\ 
N_{k}^{%
{\acute{}}%
} \\ 
B_{k}^{%
{\acute{}}%
}%
\end{array}%
\right) =\left( 
\begin{array}{ccc}
0 & \kappa _{k} & 0 \\ 
-\kappa _{k} & 0 & \tau _{k} \\ 
0 & -\tau _{k} & 0%
\end{array}%
\right) \left( 
\begin{array}{c}
T_{k} \\ 
N_{k} \\ 
B_{k}%
\end{array}%
\right)
\end{equation*}%
(\cite{Ali}). \ Furthermore we get $\kappa _{k}=\sqrt{\kappa _{k-1}^{2}+\
\tau _{k-1}^{2}}$ and $\tau _{k}=\sigma _{k-1}\kappa _{k}$ where 
\begin{equation*}
\sigma _{k-1}=\frac{\kappa _{k-1}^{2}}{\left( \kappa _{k-1}^{2}+\ \tau
_{k-1}^{2}\right) ^{\frac{3}{2}}}\left( \frac{\tau _{k-1}}{\kappa _{k-1}}%
\right) ^{\prime }
\end{equation*}%
is the geodesic curvature of spherical image of the principal normal of $%
\gamma _{k+1}$(\cite{Ali}, \cite{CY}). Centrode of the curve $\gamma _{k}$
is defined \ by 
\begin{equation*}
W_{k}(s)=\tau _{k}(s)T_{k}(s)+\kappa _{k}(s)B_{k}(s)
\end{equation*}%
for all $s\in I$. If\ \ there exist constant angle between $\psi _{k+1}(s)$
and any constant vector (i.e.\ $\gamma _{k}$ is general helix), then it is
said that $\gamma $ is $k$-slant curve (\cite{Ali}). The following equations
are equal

1) $\gamma $ \ is a $k$-slant curve. \ 

2) $\gamma _{k-1}$ \ is a slant curve (1-slant).

3) $\gamma _{k}$ \ is a general helix (0-slant).

4) $\gamma _{k+1}$ \ is a planer curve (\cite{Ali}, \cite{CY}).

\section{\textbf{The construction of all spherical k-slant curve}}

In Euclidean $3-$spaces, let $M$ be a regular spherical curve with
coordinate neighborhood $(I,\gamma )$. So we can define a curve $I(\gamma )$ 
$:I\rightarrow E^{3}$ such that 
\begin{equation}
I(\gamma )(t)=\alpha (t)=\overset{t}{\underset{0}{\dint }}S_{\gamma
}(u,\theta _{0})\gamma (u)du  \label{1}
\end{equation}%
where $S_{\gamma }:I\longrightarrow 
\mathbb{R}
$ $\left( t\rightarrow S_{\gamma }(t,\theta _{0})\right) $ is a
differentiable function.

\begin{lemma}
Under the above notation, the curve $I(\gamma )$ is a spherical curve if and
only if 
\begin{equation}
S_{\gamma }(t,\theta _{0})=\left\Vert \gamma ^{\prime }(t)\right\Vert \cos
\left( \overset{t}{\underset{0}{\dint }}\frac{\det (\gamma (u),\gamma
^{\prime }(u),\gamma ^{\prime \prime }(u))}{\left\Vert \gamma ^{\prime
}(u)\right\Vert ^{2}}du+\theta _{0}\right) .  \label{2}
\end{equation}
\end{lemma}

\begin{proof}
Without loss of generality, we suppose that the center of the sphere is
origin. If $I(\gamma )$ is a regular spherical curves, then we have%
\begin{equation*}
\left\Vert \alpha (t)\right\Vert =\left\Vert \overset{t}{\underset{0}{\dint }%
}S_{\gamma }(u)\gamma (u)du\right\Vert =1
\end{equation*}%
and $\alpha (t)$ is perpendicular to $\gamma (t)$, \ for all $s\in I$. So we
obtain 
\begin{equation}
\alpha (t)=\overset{t}{\underset{0}{\dint }}S_{\gamma }(u)\gamma
(u)du=f(t)\gamma ^{\prime }(t)+g(t)Y(t)  \label{4.1}
\end{equation}%
where $Y(t)=\gamma (t)\times \gamma ^{\prime }(t)$. \ From equation (\ref%
{4.1}), we get 
\begin{equation*}
\left( f(t)\right) ^{2}+\left( g(t)\right) ^{2}=\frac{1}{\left\Vert \gamma
^{\prime }(t)\right\Vert ^{2}}
\end{equation*}%
or 
\begin{equation}
\begin{array}{c}
f(t)=-\frac{1}{\left\Vert \gamma ^{\prime }(t)\right\Vert }\cos \theta (t)
\\ 
g(t)=\frac{1}{\left\Vert \gamma ^{\prime }(t)\right\Vert }\sin \theta (t)%
\end{array}
\label{4.4}
\end{equation}%
where $\theta :I\longrightarrow 
\mathbb{R}
$ is a function. If we derivate (\ref{4.1}), then 
\begin{equation}
S_{\gamma }(t)\gamma (t)=f^{\prime }(t)\gamma ^{\prime }(t)+f(t)\gamma
^{\prime \prime }(t)+g^{\prime }(t)\gamma (t)\times \gamma ^{\prime
}(t)+g(t)\gamma (t)\times \gamma ^{\prime \prime }(t)  \label{4.2}
\end{equation}%
From (\ref{4.4}) and (\ref{4.2}), we get 
\begin{equation}
S_{\gamma }(t)=\left\Vert \gamma ^{\prime }(t)\right\Vert \cos \theta (t)
\label{4.5}
\end{equation}%
Using (\ref{4.4}), \ we obtain%
\begin{equation}
f^{\prime }(t)=\frac{\left\langle \gamma ^{\prime }(t),\gamma ^{\prime
\prime }(t)\right\rangle }{\left\Vert \gamma ^{\prime }(t)\right\Vert ^{3}}%
\cos \theta (t)+\frac{\theta ^{\prime }(t)}{\left\Vert \gamma ^{\prime
}(t)\right\Vert }\sin \theta (t)  \label{4.6}
\end{equation}%
From (\ref{4.2}), \ we have%
\begin{equation}
f^{\prime }(t)\left\Vert \gamma ^{\prime }(t)\right\Vert
^{2}+f(t)\left\langle \gamma ^{\prime }(t),\gamma ^{\prime \prime
}(t)\right\rangle -g(t)\det (\gamma (t),\gamma ^{\prime }(t),\gamma ^{\prime
\prime }(t))=0  \label{e}
\end{equation}%
Using (\ref{4.4}), \ (\ref{4.6}) and (\ref{e}) \ we get 
\begin{equation}
\theta ^{\prime }(t)\left\Vert \gamma ^{\prime }(t)\right\Vert -\frac{\det
(\gamma (t),\gamma ^{\prime }(t),\gamma ^{\prime \prime }(t))}{\left\Vert
\gamma ^{\prime }(t)\right\Vert }=0  \label{4.7}
\end{equation}%
If we integrate (\ref{4.7}), then 
\begin{equation*}
\theta (t)=\overset{t}{\underset{0}{\dint }}\frac{\det (\gamma (u),\gamma
^{\prime }(u),\gamma ^{\prime \prime }(u))}{\left\Vert \gamma ^{\prime
}(u)\right\Vert ^{2}}du+\theta _{0}.
\end{equation*}%
Using (\ref{4.5}), we get%
\begin{equation*}
S_{\gamma }(t)=\left\Vert \gamma ^{\prime }(t)\right\Vert \cos \left( 
\overset{t}{\underset{0}{\dint }}\frac{\det (\gamma (u),\gamma ^{\prime
}(u),\gamma ^{\prime \prime }(u))}{\left\Vert \gamma ^{\prime
}(u)\right\Vert ^{2}}du+\theta _{0}\right) .
\end{equation*}
\end{proof}

\begin{corollary}
Let $SC$ be a set of spherical regular curves. From Lemma 3.1, we can define
a map as $I:SC\times \lbrack 0,2\pi ]\longrightarrow SC,I(\gamma ,\theta
_{0})=I(\gamma )(t,\theta _{0})$ where 
\begin{equation*}
S_{_{I^{0}(\gamma )}}(t,\theta _{0})=\left\Vert \gamma ^{\prime
}(t)\right\Vert \cos \left( \overset{t}{\underset{0}{\dint }}\frac{\det
(\gamma (u),\gamma ^{\prime }(u),\gamma ^{\prime \prime }(u))}{\left\Vert
\gamma ^{\prime }(u)\right\Vert ^{2}}du+\theta _{0}\right) .
\end{equation*}%
and 
\begin{equation*}
I(\gamma )(t,\theta _{0})=\alpha (t)=\overset{t}{\underset{0}{\dint }}%
S_{_{I^{0}(\gamma )}}(u,\theta _{0})I^{0}(\gamma )(u)du.
\end{equation*}%
where $I^{0}(\gamma )=\gamma $.So we get%
\begin{equation*}
S_{I(\gamma )}(t,\widetilde{\theta }_{1})=\left\Vert I(\gamma )^{\prime
}(t)\right\Vert \cos \left( \overset{t}{\underset{0}{\dint }}\frac{\det
(I(\gamma )(u),I(\gamma )^{\prime }(u),I(\gamma )^{\prime \prime }(u))}{%
\left\Vert I(\gamma )^{\prime }(u)\right\Vert ^{2}}du+\theta _{1}\right)
\end{equation*}%
and 
\begin{equation*}
I(I(\gamma )(t,\theta _{0}),\theta _{1})=I^{2}(\gamma )(t,\widetilde{\theta }%
_{1})=\overset{t}{\underset{0}{\dint }}S_{I(\gamma )}(u,\widetilde{\theta }%
_{1})I(\gamma )(u,\theta _{0})du.
\end{equation*}%
where $\widetilde{\theta }_{1}=(\theta _{0},\theta _{1})$. From induction
method, we have 
\begin{equation*}
S_{I^{n}(\gamma )}(t,\widetilde{\theta }_{n})=\left\Vert I^{n}(\gamma
)^{\prime }(t)\right\Vert \cos \left( \overset{t}{\underset{0}{\dint }}\frac{%
\det (I^{n}(\gamma )(u),I^{n}(\gamma )^{\prime }(u),I^{n}(\gamma )^{\prime
\prime }(u))}{\left\Vert I^{n}(\gamma )^{\prime }(u)\right\Vert ^{2}}%
du+\theta _{n}\right)
\end{equation*}%
and 
\begin{equation*}
I(I^{n}(\gamma )(t,\widetilde{\theta }_{n-1}),\theta _{n})=I^{n+1}(\gamma
)(s,\widetilde{\theta }_{n})=\overset{t}{\underset{0}{\dint }}%
S_{I^{n}(\gamma )}(t,\widetilde{\theta }_{n})I^{n}(\gamma )(u,\widetilde{%
\theta }_{n-1})du.
\end{equation*}%
where $\widetilde{\theta }_{n-1}=(\theta _{0},\theta _{1},...,\theta _{n-1})$
and $\widetilde{\theta }_{n}=(\theta _{0},\theta _{1},...,\theta _{n})$. If
we define 
\begin{equation*}
I^{-1}(\gamma )(t,\theta _{0})=I(-\gamma )(t,\theta _{0})=-I(\gamma
)(t,-\theta _{0}).
\end{equation*}%
then we have 
\begin{equation*}
I^{-n}(\gamma )(t,\widetilde{\theta }_{n-1})=I^{n}(-\gamma )(t,\widetilde{%
\theta }_{n-1})=-I^{n}(\gamma )(t,-\widetilde{\theta }_{n-1}).
\end{equation*}%
where $-\widetilde{\theta }_{n-1}=(-\theta _{0},-\theta _{1},...,-\theta
_{n-1})$. So we can define 
\begin{equation}
S(\gamma )=\left\{ ...,I^{-2}(\gamma ),I^{-1}(\gamma ),I^{0}(\gamma )=\gamma
,I(\gamma ),I^{2}(\gamma ),...\right\}  \label{z}
\end{equation}%
If we define a addition ana multiplication as 
\begin{equation*}
I^{n}(\gamma )+I^{m}(\gamma )=I^{n+m}(\gamma )
\end{equation*}%
\begin{equation*}
I^{n}(\gamma ).I^{m}(\gamma )=I^{nm}(\gamma )
\end{equation*}%
where $n,m\in 
\mathbb{Z}
$, then we can see that $\left( 
\mathbb{Z}
(\gamma ),+,.\right) $ is a ring.

\begin{theorem}
In Euclidean $3-$spaces, the curve $\gamma $ is a spherical $k-$slant curve
if and only if $\ I(\gamma )$ is a spherical $k+1-$slant curve.
\end{theorem}
\end{corollary}

\begin{proof}
Let $\gamma $ be a regular curve with coordinate neighborhood $(I,\gamma )$.
So we can see that 
\begin{equation*}
\psi _{0}(t)=\gamma (t)=\frac{\alpha ^{\prime }(t)}{S_{M}(t)}=\overline{\psi 
}_{1}(t)
\end{equation*}%
Thus, $\gamma $ is tangent indicatrix of the curve $I(\gamma )$ and we have $%
\psi _{k}=\overline{\psi }_{k+1}$ where $\psi _{k}(t)=\frac{\psi
_{k-1}^{\prime }(t)}{\left\Vert \psi _{k-1}^{\prime }(t)\right\Vert }$ and $%
\overline{\psi }_{k+1}(t)=\frac{\overline{\psi }_{k}^{\prime }(t)}{%
\left\Vert \overline{\psi }_{k}^{\prime }(t)\right\Vert }$.
\end{proof}

\begin{example}
If $\gamma (t)=(-\frac{1}{\sqrt{2}}\sin 2t,\frac{1}{\sqrt{2}}\sin 2t,\cos
2t) $, then we have $\left\Vert \gamma ^{\prime }(t)\right\Vert =2$, $\ \det
(\gamma (t),\gamma ^{\prime }(t),\gamma ^{\prime \prime }(t))=0$, $S_{\gamma
}(t)=2\cos (\theta _{0})$ and 
\begin{equation}
I(\gamma )(t,\theta _{0})=\int 2\cos (\theta _{0})(-\frac{1}{\sqrt{2}}\sin
2t,\frac{1}{\sqrt{2}}\sin 2t,\cos 2t)dt  \label{a}
\end{equation}%
We can see that the geodesic circle $\gamma (t)$ is the intersection of the
plane ( $x+y=0$ ) and the unit sphere $S^{2}(O,1)$ where the center of the
sphere is of origin. By integrating (\ref{a}), we obtain 
\begin{equation*}
I(\gamma )(t,\theta _{0})=\cos (\theta _{0})\left( \frac{1}{\sqrt{2}}\cos
(2t)+c_{1},-\frac{1}{\sqrt{2}}\cos (2t)+c_{2},\sin (2t)+c_{3}\right)
\end{equation*}%
In this case, $I(\gamma )$ lies on $S^{2}(O,1)$ if and only if $c_{1}=c_{2}= 
$ $c_{3}=0$ and $\cos (\theta _{0})=\varepsilon =\pm 1$. So we have a curve $%
I(\gamma )$ as 
\begin{equation*}
I(\gamma )(t)=I(\gamma )(t,0)=\left( \frac{\varepsilon }{\sqrt{2}}\cos (2t),-%
\frac{\varepsilon }{\sqrt{2}}\cos (2t),\varepsilon \sin (2t)\right) .
\end{equation*}%
We can see that the geodesic circle $I(\gamma )$ is the intersection of the
plane ( $x+y=0$ ) and the sphere $S^{2}(O,1)$ and $\left\langle \gamma
(t),I(\gamma )(t)\right\rangle =0$. So, there are four types of arc which
belong to a geodesic circle, i) $I(\gamma )=\gamma $ , ii) $I^{2}(\gamma
)=\gamma $ , iii) $I^{3}(\gamma )=\gamma $, vi)$I^{4}(\gamma )=\gamma $.
\end{example}

\begin{example}
In the first step, let $\gamma =S^{1}(\overrightarrow{a},r)$ be a circle in
sphere where $\gamma (s)=(r\cos ws,r\sin ws,a)$, $w=\frac{1}{r}$ ,$%
a=\left\Vert \overrightarrow{a}\right\Vert $ and $a^{2}+r^{2}=1$. So we get 
\begin{equation*}
S_{\gamma }(s,\theta _{0})=\cos (aws+\theta _{0})
\end{equation*}%
and 
\begin{equation*}
I(\gamma )(s,\theta _{0})=\int \cos (aws+\theta _{0})(r\cos ws,r\sin ws,a)ds
\end{equation*}%
where $\left\Vert \gamma ^{^{\prime }}(s)\right\Vert =r$ and $\ \det (\gamma
(s),\gamma ^{^{\prime }}(s),\gamma ^{^{\prime \prime }}(s))=aw$. Then we
have a curve $I(\gamma )$ as follows 
\begin{equation*}
I(\gamma )(s,\theta _{0})=\left( 
\begin{array}{c}
\frac{r}{2}\left[ \frac{1}{w(a+1)}\sin \left( w(a+1)s+\theta _{0}\right) +%
\frac{1}{w(a-1)}\sin \left( w(a-1)s+\theta _{0}\right) \right] , \\ 
\frac{r}{2}\left[ -\frac{1}{w(a+1)}\cos \left( w(a+1)s+\theta _{0}\right) +%
\frac{1}{w(a-1)}\cos \left( w(a-1)s+\theta _{0}\right) \right] , \\ 
r\sin \left( aws+\theta _{0}\right) 
\end{array}%
\right) 
\end{equation*}%
where $\left\Vert I(\gamma )(s,\theta _{0})\right\Vert =1$. So we can see
that this curve is spherical helix. Furthermore, since tangent indicatries
of $I(\gamma )(s,\theta _{0})$ are equal to $\gamma =S^{1}(a,r)$, all
spherical helices are $I(\gamma )(s,\theta _{0})$ in which axis of the helix
is equal to $\overrightarrow{u}=(0,0,1)$. Blaschke determined that all
spherical helices are equal to $I(\gamma )(s,\theta _{0})$ in which axis of
the helix is equal to $\overrightarrow{u}=(0,0,1)$(\cite{BL}, \cite{ST}) .
Furthermore, Blaschke gave that projections of the spherical helices onto $%
xOy$ plane are arcs of an epicycloids(\cite{BL}, \cite{ST}) . In the second
step, from (\ref{2}), \ we get 
\begin{equation*}
S_{I(\gamma )}(s,\widetilde{\theta }_{1})=\cos (aws+\theta _{1})\cos \left( 
\frac{1}{aw}\cos (aws+\theta _{1})+\theta _{2}\right) .
\end{equation*}%
So we have a curve$I(I(\gamma )(s,\theta _{0}),\theta _{1})=I^{2}(\gamma )(s,%
\widetilde{\theta }_{1})$ as follows 
\begin{equation*}
I^{2}(\gamma )(s,\widetilde{\theta }_{1})=\int S_{I(\gamma )}(s,\widetilde{%
\theta }_{1})I(\gamma )(s,\widetilde{\theta }_{0})ds.
\end{equation*}

Since tangent indicatries of $I^{2}(\gamma )(s,\widetilde{\theta }_{1})$ are
equal to $I(\gamma )(s,\widetilde{\theta }_{0})$, all spherical slant curve
are $I^{2}(\gamma )(s,\widetilde{\theta }_{1})$ in which axis of the slant
curve is equal to $\overrightarrow{u}=(0,0,1)$. In the third step, let $%
I^{2}(\gamma )$ be a spherical curve where 
\begin{equation*}
I^{2}(\gamma )(s,\widetilde{\theta }_{1})=\int S_{I(\gamma )}(s,\widetilde{%
\theta }_{1})I(\gamma )(s,\widetilde{\theta }_{0})ds
\end{equation*}%
then we get%
\begin{eqnarray*}
S_{I^{2}(\gamma )}(s,\widetilde{\theta }_{2}) &=&\cos (aws+\theta _{0})\cos
\left( \frac{1}{aw}\cos (aws+\theta _{0})+\theta _{1}\right) \\
&&\times \cos \left( \int \det (\gamma (s),\alpha (s),\alpha ^{\prime
}(s))ds+\theta _{2}\right)
\end{eqnarray*}%
Thus we have a curve $C^{2}(a,r)$ as follows 
\begin{equation*}
I^{3}(\gamma )(s,\widetilde{\theta }_{2})=\int S_{I^{2}(\gamma )}(s,%
\widetilde{\theta }_{2})I^{2}(\gamma )(s,\widetilde{\theta }_{1})ds.
\end{equation*}%
This curve is a spherical $2-$slant curve. Since tangent indicatries of $%
I^{3}(\gamma )$ is equal to $I^{2}(\gamma )$, all spherical $2-$slant curve
is $I^{3}(\gamma )$ in which axis of the $2-$slant curve is equal to $%
\overrightarrow{u}=(0,0,1)$. By the induction method, \ we have all $3-$%
slant curves $I^{4}(\gamma )$, all $4-$slant curves $I^{5}(\gamma )$, all $%
5- $slant curves $I^{6}(\gamma )$,...,all $k-$slant curves $I^{k+1}(\gamma )$%
,..., where axis of the $i-$slant curve is equal to$\overrightarrow{u}%
=(0,0,1)$, for all $i\in 
\mathbb{N}
$. \ So we have following corollary.
\end{example}

\begin{corollary}
In Euclidean $3-$spaces, there exists spherical $k$-slant curve, for all $%
k\in 
\mathbb{N}
$ where $%
\mathbb{N}
$ is a set of the natural numbers. Furthermore, all spherical $k-$slant
curves are $I^{k+2}(\gamma )$ in which axis of the $k-$slant curve is equal
to $\overrightarrow{u}=(0,0,1)$, for all $k\in 
\mathbb{N}
$. If the geodesic circle which lies on sphere are changed, then we have all
spherical $k-$slant curves.

\begin{definition}
In Euclidean $3-$spaces, we can define a set by%
\begin{equation*}
S\left( S^{1}(\overrightarrow{a},r)\right) =\left\{ ...,I^{-2}(\gamma
),I^{-1}(\gamma ),\gamma =S^{1}(\overrightarrow{a},r),I(\gamma
),I^{2}(\gamma ),I^{3}(\gamma ),...\right\}
\end{equation*}%
We say that $S\left( S^{1}(\overrightarrow{a},r)\right) $ is a set of
spherical slant curve. So, the set of all spherical $k$-slant curve $(S^{I})$
is given by%
\begin{equation*}
S^{I}=\underset{\overrightarrow{a}\in D}{\cup }S\left( S^{1}(\overrightarrow{%
a},r)\right)
\end{equation*}%
where $D=\left\{ \overrightarrow{a}%
=(a_{1},a_{2},a_{3})/a_{1}^{2}+a_{2}^{2}+a_{3}^{3}\langle 1\right\} $, $%
\left\Vert \overrightarrow{a}\right\Vert ^{2}+r^{2}=1$.
\end{definition}
\end{corollary}

The curve lies on the 2-sphere if and only if 
\begin{equation}
\left( \left( \frac{1}{\kappa }\right) ^{\prime }\frac{1}{\tau }\right)
^{\prime }+\frac{\tau }{\kappa }=0  \label{1.1}
\end{equation}%
(\cite{EK}). The solution of the equation (\ref{1.1}) is given by 
\begin{equation}
\frac{1}{\kappa }=A\cos \left( \overset{s}{\underset{0}{\dint }}\tau
(u)du\right) +B\cos \left( \overset{s}{\underset{0}{\dint }}\tau (u)du\right)
\label{1.2}
\end{equation}%
where $R=\sqrt{A^{2}+B^{2}}$ is the radius of a sphere (\cite{BG}, \cite{EK}%
, \cite{wong}). From equation (\ref{1.2}), we have 
\begin{equation}
\frac{1}{\kappa }=R\left( \frac{A}{R}\cos \left( \overset{s}{\underset{0}{%
\dint }}\tau (u)du\right) +\frac{B}{R}\cos \left( \overset{s}{\underset{0}{%
\dint }}\tau (u)du\right) \right)  \label{1.3}
\end{equation}%
If $\cos \alpha _{0}=\frac{A}{R}$, then we have $\sin \alpha _{0}=\frac{B}{R}
$. From equation (\ref{1.3}), \ we get 
\begin{equation}
\frac{1}{\kappa }=R\cos \left( \overset{s}{\underset{0}{\dint }}\tau
(u)du+\alpha _{0}\right)  \label{1.4}
\end{equation}%
Let $M$ be a unit-speed regular curve with coordinate neighborhood $%
(I,\gamma )$. In this paper, we suppose $0\in I$ without loss of generality.
For all $s\in I$, the oscillating sphere of the curve is equal to the sphere
in which the curve lies on sphere. Furthermore, the oscillating circle lies
on this sphere. From equation (\ref{1.4}), we get $\frac{1}{\kappa _{0}}%
=\cos \alpha _{0}=\frac{R_{0}}{R}.$ If $R$ is equal to $1$, then we have%
\begin{equation}
\cos \alpha _{0}=R_{0}=\frac{1}{\kappa _{0}}  \label{10}
\end{equation}%
where $R_{0}=\frac{1}{\kappa _{0}}=\sup \left\{ \frac{1}{\kappa (s)}/s\in
I\right\} $. So we can provide another proof of characterization of the
spherical curve.

\begin{proof}
Let $K$ be a regular spherical curve with coordinate neighborhood $(I,\beta
) $ and "$s$" be arc-length parameter of the curve and $\kappa $, \ $\tau $
be curvatures of the curves. So we can define a spherical curve as $%
S_{\gamma }(s)\gamma (s)=\beta ^{\prime }(s)$ where 
\begin{equation*}
S_{\gamma }(s)=\left\Vert \gamma ^{\prime }(s)\right\Vert \cos \left( 
\overset{s}{\underset{0}{\dint }}\frac{\det (\gamma (u),\gamma ^{\prime
}(u),\gamma ^{\prime \prime }(u))}{\left\Vert \gamma ^{\prime
}(u)\right\Vert ^{2}}du+\theta _{0}\right) .
\end{equation*}%
Because of $\ S_{\gamma }(s)\gamma (s)=\beta ^{\prime }(s)$, we have $%
S_{\gamma }(s)=1$ and $\gamma (s)=\beta ^{\prime }(s)$. Then we have%
\begin{equation}
1=\kappa \cos \left( \overset{s}{\underset{0}{\dint }}\tau (u)du+\theta
_{0}\right)  \label{7}
\end{equation}%
Conversely, we suppose that the curve $K$ satisfy equation (\ref{7}). Let $K$
be a unit-speed regular curve with coordinate neighborhood $(I,\gamma )$. We
can define a curve $M$ with coordinate neighborhood $(I,\gamma )$ where $%
\gamma (s)=\beta ^{\prime }(s)$. From equation (\ref{7}), \ we get $%
S_{\gamma }(s)=1$ and $\int S_{\gamma }(s)\gamma (s)ds=\int \beta ^{\prime
}(s)ds=\beta (s)$. From Lemma 3.1, $K$ is a regular spherical curve.
\end{proof}

\begin{definition}
In $3$-Euclidean spaces, if a spherical curve $\gamma $ does not exist such
that $I(\gamma )(t,0)=\alpha $ , then it is said that $\alpha $ is a prime
spherical curve.

\begin{corollary}
In $3$-Euclidean spaces, $S^{1}(\overrightarrow{a},r)$ are the prime
spherical curves where $r\rangle 0$ and $\left\Vert \overrightarrow{a}%
\right\Vert ^{2}+r^{2}=1$.

\begin{definition}
In $3$-Euclidean spaces, let $M$ be a regular spherical curve with
coordinate neighborhood $(I,\gamma )$. If $M$ is prime spherical curve, then
we have%
\begin{equation*}
S(\gamma )=\left\{ ...,I^{-2}(\gamma ),I^{-1}(\gamma ),\gamma ,I(\gamma
),I^{2}(\gamma ),I^{3}(\gamma ),...\right\} .
\end{equation*}%
It is said that $S(\gamma )$ is a spherical curve chain generated by $\gamma 
$. We can see that $S(\gamma )=S(-\gamma )$.

\begin{theorem}
In $3$-Euclidean spaces, let $K$ be a regular curve with coordinate
neighborhood $(I,\beta )$ and "$s$" be arc-length parameter of the curve.
For all $s\in I$, \ the oscillating circle of $K$ is not a geodesic circle
on sphere if and only if $K$ is prime spherical curve.
\end{theorem}
\end{definition}
\end{corollary}
\end{definition}

\begin{proof}
Let "$s$" be arc-lenght parameter and $\kappa $, \ $\tau $ be curvatures of
the curves $K$. We suppose that the oscillating circle of $K$ is not a
geodesic circle on sphere, for all $s\in I$. If $K$ is not a prime spherical
curve, then a curve $M$ \ exists such that $I(\gamma )(t,0)=\alpha $ where $%
M $ is a regular curve with coordinate neighborhood $(I,\gamma )$. From
Lemma 3.1, we have $S_{\gamma }(s)=1$ and $\gamma (s)=\beta ^{\prime }(s)$.
From equation (\ref{1.4}), \ we get 
\begin{equation}
\frac{1}{\kappa }=\cos \left( \overset{s}{\underset{0}{\dint }}\tau
(u)du\right) .  \label{8}
\end{equation}%
From equation (\ref{8}), \ we have $\cos \left( 0\right) =R_{0}=\frac{1}{%
\kappa _{0}}=1$. This is contradiction. Conversely, we suppose that $K$ \ is
a prime spherical curve. If we define a curve $M$ such that $\gamma
(s)=\beta ^{\prime }(s)$, then we have 
\begin{equation}
S_{\gamma }(s)=1=\kappa \cos \left( \overset{s}{\underset{0}{\dint }}\tau
(u)du+\theta _{0}\right)  \label{9}
\end{equation}%
From equation (\ref{9}), \ we can see that $\cos \left( \theta _{0}\right)
=R_{0}=\frac{1}{\kappa _{0}}\neq 1$.
\end{proof}

\section{\textbf{The construction of all k-slant curve in E}$^{3}$}

We can apply Lemma 3.1 into curve theory. Let $M$ be a unit-speed curve with
coordinate neighborhood $(I,\gamma )$ and $\left\{ T,N,B,\kappa ,\ \tau
\right\} $ be Serret-Frenet apparatus of the curves. So we have $\left\Vert
\gamma ^{\prime }(s)\right\Vert =1$ and tangent indicatrix of curve $M$ is $%
\sigma (s)=\gamma ^{\prime }(s)=T(s)$. There exist $S_{T}:I\longrightarrow 
\mathbb{R}
$ differentiable function such that 
\begin{equation*}
\left\Vert \int S_{T}(s)\gamma ^{\prime }(s)ds\right\Vert =1
\end{equation*}%
where 
\begin{equation}
S_{T}(s)=\kappa (s)\cos \left( \overset{s}{\underset{0}{\dint }}\tau
(u)du+\theta _{0}\right) .  \label{4.8}
\end{equation}%
In this case, we can define a unit-speed curve $K$ with coordinate
neighborhood $(I,\beta )$ such that 
\begin{equation*}
I(D\gamma )(s,\theta _{0})=\beta ^{\prime }(s)=\int S_{T}(s)\gamma ^{\prime
}(s)ds.
\end{equation*}%
and 
\begin{equation}
\beta ^{\prime \prime }(s)=S_{T}(s)\gamma ^{\prime }(s)  \label{4.9}
\end{equation}%
So the curve $K$ get 
\begin{equation*}
\beta (s)=D^{-1}I(D\gamma )(s,\theta _{0})=J(\gamma )(s,\theta _{0})
\end{equation*}%
Let $\overline{T},\overline{N},\overline{B}$ be Serret-Frenet vectors and $%
\overline{\kappa },\overline{\tau }$ be curvature and torsion of the curve $%
K $, respectively, where 
\begin{equation*}
\overline{\kappa }(s)=S_{T}(s)=\kappa (s)\cos \left( \overset{s}{\underset{0}%
{\dint }}\tau (u)du+\theta _{0}\right)
\end{equation*}%
and 
\begin{equation}
\overline{\tau }(s)=\kappa (s)\sin \left( \overset{s}{\underset{0}{\dint }}%
\tau (u)du+\theta _{0}\right) \text{.}  \label{4.10}
\end{equation}%
From (\ref{4.9}), we have 
\begin{equation}
\overline{N}(s)=\varepsilon T(s)  \label{4.11}
\end{equation}%
where $\varepsilon =\pm 1$. Without loss of generality, we suppose that%
\begin{equation*}
\overline{\kappa }(s)=\kappa (s)\cos \left( \overset{s}{\underset{0}{\dint }}%
\tau (u)du+\theta _{0}\right)
\end{equation*}%
From equation (\ref{4.9}) and (\ref{4.11}), we can see that principal normal
of $K$ and tangent of $M$ is colinear.

\begin{theorem}
In 3-Euclidean spaces, let $\gamma $ be a unit speed curves with coordinate
neighborhoods $(I,\gamma )$. In this case, $\gamma $ is a $k$-slant curve if
and only if $J(\gamma )$ is a $k+1$-slant curve. The proof is the same as
the theorem 3.1. So we can see that $i$) $\gamma $ is a planar curve if and
only if $J(\gamma )$ is a general helix $\ ii$) $\gamma $ is a general helix
if and only if $J(\gamma )$ is a slant curve.
\end{theorem}

\begin{definition}
In 3-Euclidean spaces, let $M$ be a unit-speed regular curve with coordinate
neighborhood $(I,\gamma )$. From equation (\ref{4.9}),if the curve is planar
curve, we have helix $\left( J(\gamma )\right) $ in first step. In second
step, we have $1-$slant curve(slant curve). In third step, we have $2-$slant
curve $\left( J^{2}(\gamma )\right) $. ... In $k+1$. step, we have k-slant
curve$\left( J^{k+1}(\gamma )\right) $. So we have a set 
\begin{equation*}
\mathbb{Z}
(\gamma )=\left\{ ...,J^{-2}(\gamma ),J^{-1}(\gamma ),\gamma ,J(\gamma
),J^{2}(\gamma ),...\right\} .
\end{equation*}%
We said that this set is a slant curve chain taken by a planer curve $\gamma 
$.
\end{definition}

\begin{example}
In 3-Euclidean spaces, let $\gamma =S^{1}$ be a circle where $\gamma
(s)=(r\cos ws,-r\sin ws,0)$ and $w=\frac{1}{r}$. In first step, $\kappa =w$
and $\tau =0$. So we have $\overline{\kappa }(s)=\epsilon w\cos
c_{0}=A(const.)$ and $\overline{\tau }(s)=\overline{\epsilon }w\sin
c_{0}=B(const.)$ and 
\begin{equation}
\beta ^{\prime \prime }(s)=A(-\sin ws,-\cos ws,0)  \label{4.22}
\end{equation}%
When we integrate of equation (\ref{4.22}), we have a curve $M_{1}=\left(
S^{1}\right) _{1}$ as 
\begin{equation*}
\beta (s)=(Ar^{2}\cos \left( \epsilon ws\right) ,Ar^{2}\sin \left( \epsilon
ws\right) ,bws+c_{1})
\end{equation*}%
Because $M_{1}=\left( S^{1}\right) _{1}$ is unit spreed curve, we have $%
b=r\sin c_{0}$. If $a=Ar^{2}=\epsilon r\cos c_{0}$ and $c_{1}=0$, we have%
\begin{equation*}
J(\gamma )(s)=(a\cos \left( ws\right) ,a\sin \left( ws\right) ,bws)
\end{equation*}%
where $r=\sqrt{a^{2}+b^{2}}$.
\end{example}

\begin{example}
In second step, if $J(\gamma )(s)=(a\cos ws,a\sin ws,bws)$, then we have $%
\kappa =aw^{2}$ and $\tau =bw^{2}$. \ From equation (\ref{4.22}), we have $%
\overline{\kappa }(s)=\epsilon aw^{2}\cos \left( bw^{2}s\right) $ and $%
\overline{\tau }(s)=\overline{\epsilon }aw^{2}\sin \left( bw^{2}s\right) $.
So we have%
\begin{equation*}
J^{2}(\gamma )^{\prime \prime }(s)=\epsilon aw^{2}\cos \left( bw^{2}s\right)
\left( -aw\sin ws,aw\cos ws,bw\right)
\end{equation*}%
and 
\begin{equation}
J^{2}(\gamma )^{\prime \prime }(s)=\left( -\epsilon a^{2}w^{3}\cos \left(
bw^{2}s\right) \sin ws,\epsilon a^{2}w^{3}\cos \left( bw^{2}s\right) \cos
ws,\epsilon abw^{3}\cos \left( bw^{2}s\right) \right)  \label{4.23}
\end{equation}%
When we integrate of equation (\ref{4.23}), we get 
\begin{equation}
J^{2}(\gamma )^{\prime }(s)=\epsilon aw^{2}\left( 
\begin{array}{c}
\frac{a}{2}\left[ \frac{1}{1+bw}\cos \left( w(1+bw)s\right) +\frac{1}{1-bw}%
\cos \left( w(1-bw)s\right) \right] , \\ 
\frac{a}{2}\left[ \frac{1}{1+bw}\sin \left( w(1+bw)s\right) +\frac{1}{1-bw}%
\sin \left( w(1-bw)s\right) \right] , \\ 
\frac{1}{w}\sin \left( bw^{2}s\right)%
\end{array}%
\right)  \label{4.24}
\end{equation}%
So we can see that $\left\Vert J^{2}(\gamma )^{\prime }(s)\right\Vert =1$.
If we integrate the equation (\ref{4.24}), we have a curve $J^{2}(\gamma )$
as $J^{2}(\gamma )(s)=\left( x(s),y(s),z(s)\right) $ where 
\begin{equation}
\left( 
\begin{array}{c}
x(s)=\frac{\epsilon a^{2}w}{2}\left[ \frac{1}{\left( 1+bw\right) ^{2}}\sin
\left( w(1+bw)s\right) +\frac{1}{\left( 1-bw\right) ^{2}}\sin \left(
w(1-bw)s\right) \right] \\ 
y(s)=-\frac{\epsilon a^{2}w}{2}\left[ \frac{1}{\left( 1+bw\right) ^{2}}\cos
\left( w(1+bw)s\right) +\frac{1}{\left( 1-bw\right) ^{2}}\cos \left(
w(1-bw)s\right) \right] \\ 
z(s)=-\frac{\epsilon a}{bw}\cos \left( bw^{2}s\right)%
\end{array}%
\right)  \label{4.25}
\end{equation}%
Thus we get 
\begin{equation}
x^{2}+y^{2}-\frac{b^{2}}{a^{2}}z^{2}=\frac{b^{2}}{a^{4}w^{4}}  \label{4.26}
\end{equation}%
If $\gamma $ is a unit circle $\left( r=1\right) $, then we have $w=\frac{1}{%
r}=1$. \ From (\ref{4.26}), we have%
\begin{equation}
x^{2}+y^{2}-\frac{b^{2}}{a^{2}}z^{2}=\frac{b^{2}}{a^{4}}  \label{4.27}
\end{equation}%
Thus we get similar solutions in (\cite{RB}, \cite{Sco}).
\end{example}

In third step, if 
\begin{equation*}
J^{2}(\gamma )(s)=\left( 
\begin{array}{c}
\frac{a^{2}w}{2}\left[ \frac{1}{\left( 1+bw\right) ^{2}}\sin \left(
w(1+bw)s\right) +\frac{1}{\left( 1-bw\right) ^{2}}\sin \left(
w(1-bw)s\right) \right] , \\ 
-\frac{a^{2}w}{2}\left[ \frac{1}{\left( 1+bw\right) ^{2}}\cos \left(
w(1+bw)s\right) +\frac{1}{\left( 1-bw\right) ^{2}}\cos \left(
w(1-bw)s\right) \right] , \\ 
-\frac{a}{bw}\cos \left( bw^{2}s\right)%
\end{array}%
\right)
\end{equation*}%
then we have $\kappa (s)=aw^{2}\cos \left( bw^{2}s\right) $ and $\tau
(s)=aw^{2}\sin \left( bw^{2}s\right) $. From (\ref{4.22}), we get 
\begin{equation*}
\overline{\kappa }(s)=\epsilon aw^{2}\cos \left( bw^{2}s\right) \cos \left( 
\frac{a}{b}\cos \left( bw^{2}s\right) \right)
\end{equation*}%
and 
\begin{equation*}
\overline{\tau }(s)=\overline{\epsilon }aw^{2}\cos \left( bw^{2}s\right)
\sin \left( \frac{a}{b}\cos \left( bw^{2}s\right) \right) \text{.}
\end{equation*}%
So we have%
\begin{equation*}
J^{3}(\gamma )(s)^{\prime \prime }(s)=\overline{\kappa }(s)\left( 
\begin{array}{c}
\frac{a^{2}w^{2}}{2}\left[ \frac{1}{1+bw}\cos \left( w(1+bw)s\right) +\frac{1%
}{1-bw}\cos \left( w(1-bw)s\right) \right] , \\ 
\frac{a^{2}w^{2}}{2}\left[ \frac{1}{1+bw}\sin \left( w(1+bw)s\right) +\frac{1%
}{1-bw}\sin \left( w(1-bw)s\right) \right] , \\ 
\frac{1}{w}\sin \left( bw^{2}s\right)%
\end{array}%
\right)
\end{equation*}%
Furthermore, from (\cite{WWB}, p.101), if $t=e^{i\left( \frac{\pi }{2}-\phi
\right) }$, then we have%
\begin{equation*}
\cos (x\cos \phi )=J_{0}(x)+2\overset{\infty }{\underset{k=1}{\dsum }}%
(-1)^{k}J_{2k}(x)\cos (2k\phi )
\end{equation*}%
and 
\begin{equation*}
\sin (x\cos \phi )=2\overset{\infty }{\underset{k=1}{\dsum }}%
(-1)^{k}J_{2k-1}(x)\cos \left( (2k-1)\phi \right)
\end{equation*}%
where $J_{n}$ is said Bessel function and defined as 
\begin{equation*}
J_{n}(x)=\overset{\infty }{\underset{k=1}{\dsum }}(-1)^{k}\frac{\Gamma (n+1)%
}{2^{2k}k!\Gamma (n+k+1)}x^{2k+n}
\end{equation*}%
or 
\begin{equation*}
J_{n}(x)=\frac{1}{\pi }\overset{\pi }{\underset{0}{\int }}\cos \left( n\phi
-x\sin \phi \right) d\phi
\end{equation*}%
$\left( n\text{ integer}\right) $ (\cite{WWB}). \ If $\phi (s)=bw^{2}s$ and $%
x=\frac{a}{b}$, then we get 
\begin{equation}
\cos (\frac{a}{b}\cos \left( bw^{2}s\right) )=J_{0}(\frac{a}{b})+\overset{%
\infty }{\underset{k=1}{\dsum }}2(-1)^{k}J_{2k}(\frac{a}{b})\cos (2kbw^{2}s)
\label{4.28}
\end{equation}%
and%
\begin{equation}
\sin (\frac{a}{b}\cos \left( bw^{2}s\right) )=2\overset{\infty }{\underset{%
k=1}{\dsum }}(-1)^{k}J_{2k-1}(\frac{a}{b})\cos \left( (2k-1)bw^{2}s\right)
\label{4.29}
\end{equation}%
From (\ref{4.28}), we have%
\begin{equation}
\overline{\kappa }(s)=\epsilon aw^{2}\cos \left( bw^{2}s\right) \left( J_{0}(%
\frac{a}{b})+\overset{\infty }{\underset{k=1}{\dsum }}2(-1)^{k}J_{2k}(\frac{a%
}{b})\cos (2kbw^{2}s)\right)  \label{4.30}
\end{equation}%
Since $J^{3}(\gamma )(s)^{\prime \prime }(s)=\left( x^{\prime \prime
}(s),y^{\prime \prime }(s),z^{\prime \prime }(s)\right) $, then we have 
\begin{eqnarray*}
x^{\prime \prime }(s) &=&\epsilon \frac{a^{3}w^{4}}{4}\left( J_{0}(\frac{a}{b%
})+\overset{\infty }{\underset{k=1}{\dsum }}2(-1)^{k}J_{2k}(\frac{a}{b})\cos
(2kbw^{2}s)\right) . \\
&&\left( 
\begin{array}{c}
\frac{1}{1+bw}\cos \left( w(1+2bw)s\right) +\frac{1}{1-bw}\cos \left(
w(1-2bw)s\right) \\ 
+\frac{2}{a^{2}w^{2}}\cos \left( ws\right)%
\end{array}%
\right)
\end{eqnarray*}%
\begin{eqnarray*}
y^{\prime \prime }(s) &=&\epsilon \frac{a^{3}w^{4}}{4}\left( J_{0}(\frac{a}{b%
})+\overset{\infty }{\underset{k=1}{\dsum }}2(-1)^{k}J_{2k}(\frac{a}{b})\cos
(2kbw^{2}s)\right) . \\
&&\left( 
\begin{array}{c}
\frac{1}{1+bw}\sin \left( w(1+2bw)s\right) +\frac{1}{1-bw}\sin \left(
w(1-2bw)s\right) \\ 
+\frac{2}{a^{2}w^{2}}\sin \left( ws\right)%
\end{array}%
\right)
\end{eqnarray*}%
\begin{equation*}
z^{\prime \prime }(s)=\epsilon aw\cos \left( bw^{2}s\right) \sin \left(
bw^{2}s\right) \cos \left( \frac{a}{b}\cos \left( bw^{2}s\right) \right) .
\end{equation*}%
Thus we get 
\begin{equation*}
\begin{array}{c}
x^{\prime \prime }(s)=\epsilon \frac{a^{3}w^{4}}{4}J_{0}(\frac{a}{b})\left( 
\begin{array}{c}
\frac{1}{1+bw}\cos \left( w(1+2bw)s\right) +\frac{2}{a^{2}w^{2}}\cos \left(
ws\right) \\ 
+\frac{1}{1-bw}\cos \left( w(1-2bw)s\right)%
\end{array}%
\right) \\ 
+\epsilon \frac{a^{3}w^{4}}{4}\left( 
\begin{array}{c}
\frac{1}{1+bw}\overset{\infty }{\underset{k=1}{\dsum }}(-1)^{k}J_{2k}(\frac{a%
}{b})\left( 
\begin{array}{c}
\cos \left( \left( 2bw\left( k+1\right) +1\right) ws\right) \\ 
+\cos \left( \left( 2bw\left( k-1\right) -1\right) ws\right)%
\end{array}%
\right) \\ 
+\frac{1}{1-bw}\overset{\infty }{\underset{k=1}{\dsum }}(-1)^{k}J_{2k}(\frac{%
a}{b})\left( 
\begin{array}{c}
\cos \left( \left( 2bw\left( k-1\right) +1\right) ws\right) \\ 
+\cos \left( \left( 2bw\left( k+1\right) -1\right) ws\right)%
\end{array}%
\right) \\ 
+\frac{1}{a^{2}w^{2}}\overset{\infty }{\underset{k=1}{\dsum }}%
2(-1)^{k}J_{2k}(\frac{a}{b})\left( 
\begin{array}{c}
\cos \left( \left( 2kbw+1\right) ws\right) \\ 
+\cos \left( \left( 2kbw-1\right) ws\right)%
\end{array}%
\right)%
\end{array}%
\right)%
\end{array}%
\end{equation*}%
\begin{equation*}
\begin{array}{c}
y^{\prime \prime }(s)=\epsilon \frac{a^{3}w^{4}}{4}J_{0}(\frac{a}{b})\left( 
\begin{array}{c}
\frac{1}{1+bw}\sin \left( w(1+2bw)s\right) +\frac{2}{a^{2}w^{2}}\sin \left(
ws\right) \\ 
+\frac{1}{1-bw}\sin \left( w(1-2bw)s\right)%
\end{array}%
\right) \\ 
+\epsilon \frac{a^{3}w^{4}}{4}\left( 
\begin{array}{c}
\frac{1}{1+bw}\overset{\infty }{\underset{k=1}{\dsum }}(-1)^{k}J_{2k}(\frac{a%
}{b})\left( 
\begin{array}{c}
\sin \left( \left( 2bw\left( k+1\right) +1\right) ws\right) \\ 
+\sin \left( \left( 2bw\left( k-1\right) -1\right) ws\right)%
\end{array}%
\right) \\ 
+\frac{1}{1-bw}\overset{\infty }{\underset{k=1}{\dsum }}(-1)^{k}J_{2k}(\frac{%
a}{b})\left( 
\begin{array}{c}
\sin \left( \left( 2bw\left( k-1\right) +1\right) ws\right) \\ 
+\sin \left( \left( 2bw\left( k+1\right) -1\right) ws\right)%
\end{array}%
\right) \\ 
+\frac{1}{a^{2}w^{2}}\overset{\infty }{\underset{k=1}{\dsum }}(-1)^{k}J_{2k}(%
\frac{a}{b})\left( 
\begin{array}{c}
\sin \left( \left( 2kbw+1\right) ws\right) \\ 
+\sin \left( \left( 2kbw-1\right) ws\right)%
\end{array}%
\right)%
\end{array}%
\right)%
\end{array}%
\end{equation*}%
\begin{equation*}
z^{\prime \prime }(s)=\epsilon aw\cos \left( bw^{2}s\right) \sin \left(
bw^{2}s\right) \cos \left( \frac{a}{b}\cos \left( bw^{2}s\right) \right) .
\end{equation*}%
If we integrate blow equation, we have 
\begin{equation*}
\begin{array}{c}
x^{\prime }(s)=\epsilon \frac{a^{3}w^{3}}{4}J_{0}(\frac{a}{b})\left( 
\begin{array}{c}
\frac{1}{\left( 1+bw\right) (1+2bw)}\sin \left( w(1+2bw)s\right) +\frac{2}{%
a^{2}w^{2}}\sin \left( ws\right) \\ 
+\frac{1}{\left( 1-bw\right) (1-2bw)}\sin \left( w(1-2bw)s\right)%
\end{array}%
\right) \\ 
+\epsilon \frac{a^{3}w^{3}}{4}\left( 
\begin{array}{c}
\frac{1}{1+bw}\overset{\infty }{\underset{k=1}{\dsum }}(-1)^{k}J_{2k}(\frac{a%
}{b})\left( 
\begin{array}{c}
\frac{1}{\left( 2bw\left( k+1\right) +1\right) }\sin \left( \left( 2bw\left(
k+1\right) +1\right) ws\right) \\ 
+\frac{1}{\left( 2bw\left( k-1\right) -1\right) }\sin \left( \left(
2bw\left( k-1\right) -1\right) ws\right)%
\end{array}%
\right) \\ 
+\frac{1}{1-bw}\overset{\infty }{\underset{k=1}{\dsum }}(-1)^{k}J_{2k}(\frac{%
a}{b})\left( 
\begin{array}{c}
\frac{1}{\left( 2bw\left( k-1\right) +1\right) }\sin \left( \left( 2bw\left(
k-1\right) +1\right) ws\right) \\ 
+\frac{1}{\left( 2bw\left( k+1\right) -1\right) }\sin \left( \left(
2bw\left( k+1\right) -1\right) ws\right)%
\end{array}%
\right) \\ 
+\frac{1}{a^{2}w^{2}}\overset{\infty }{\underset{k=1}{\dsum }}%
2(-1)^{k}J_{2k}(\frac{a}{b})\left( 
\begin{array}{c}
\frac{1}{\left( 2kbw+1\right) }\sin \left( \left( 2kbw+1\right) ws\right) \\ 
+\frac{1}{\left( 2kbw-1\right) }\sin \left( \left( 2kbw-1\right) ws\right)%
\end{array}%
\right)%
\end{array}%
\right)%
\end{array}%
\end{equation*}%
\begin{equation*}
\begin{array}{c}
y^{\prime }(s)=-\epsilon \frac{a^{3}w^{3}}{4}J_{0}(\frac{a}{b})\left( 
\begin{array}{c}
\frac{1}{\left( 1+bw\right) (1+2bw)}\cos \left( w(1+2bw)s\right) +\frac{2}{%
a^{2}w^{2}}\cos \left( ws\right) \\ 
+\frac{1}{\left( 1-bw\right) (1-2bw)}\cos \left( w(1-2bw)s\right)%
\end{array}%
\right) \\ 
-\epsilon \frac{a^{3}w^{3}}{4}\left( 
\begin{array}{c}
\frac{1}{1+bw}\overset{\infty }{\underset{k=1}{\dsum }}(-1)^{k}J_{2k}(\frac{a%
}{b})\left( 
\begin{array}{c}
\frac{1}{\left( 2bw\left( k+1\right) +1\right) ^{2}}\cos \left( \left(
2bw\left( k+1\right) +1\right) ws\right) \\ 
+\frac{1}{\left( 2bw\left( k-1\right) -1\right) ^{2}}\cos \left( \left(
2bw\left( k-1\right) -1\right) ws\right)%
\end{array}%
\right) \\ 
+\frac{1}{1-bw}\overset{\infty }{\underset{k=1}{\dsum }}(-1)^{k}J_{2k}(\frac{%
a}{b})\left( 
\begin{array}{c}
\frac{1}{\left( 2bw\left( k-1\right) +1\right) ^{2}}\cos \left( \left(
2bw\left( k-1\right) +1\right) ws\right) \\ 
+\frac{1}{\left( 2bw\left( k+1\right) -1\right) ^{2}}\cos \left( \left(
2bw\left( k+1\right) -1\right) ws\right)%
\end{array}%
\right) \\ 
+\frac{1}{a^{2}w^{2}}\overset{\infty }{\underset{k=1}{\dsum }}%
2(-1)^{k}J_{2k}(\frac{a}{b})\left( 
\begin{array}{c}
\frac{1}{\left( 2kbw+1\right) ^{2}}\cos \left( \left( 2kbw+1\right) ws\right)
\\ 
+\frac{1}{\left( 2kbw-1\right) ^{2}}\cos \left( \left( 2kbw-1\right)
ws\right)%
\end{array}%
\right)%
\end{array}%
\right)%
\end{array}%
\end{equation*}%
\begin{equation*}
z^{\prime }(s)=\frac{\epsilon b^{3}w^{4}}{a}\left( \cos \left( \frac{a}{b}%
\cos \left( bw^{2}s\right) \right) +\frac{a}{b}\cos \left( bw^{2}s\right)
\sin \left( \frac{a}{b}\cos \left( bw^{2}s\right) \right) \right)
\end{equation*}%
where $\left( x^{\prime }(s)\right) ^{2}+\left( y^{\prime }(s)\right)
^{2}+\left( z^{\prime }(s)\right) ^{2}=1$. Because of equation (\ref{4.28})
and (\ref{4.29}), we have 
\begin{equation*}
z^{\prime }(s)=\frac{\epsilon b^{3}w^{4}}{a}\left( 
\begin{array}{c}
J_{0}(\frac{a}{b})+\overset{\infty }{\underset{k=1}{\dsum }}2(-1)^{k}J_{2k}(%
\frac{a}{b})\cos (2kbw^{2}s) \\ 
+\frac{2a}{b}\cos \left( bw^{2}s\right) \overset{\infty }{\underset{k=1}{%
\dsum }}(-1)^{k}J_{2k-1}(\frac{a}{b})\cos \left( (2k-1)bw^{2}s\right)%
\end{array}%
\right)
\end{equation*}%
So we have 
\begin{equation*}
z^{\prime }(s)=-\frac{\epsilon b}{aw}\left( 
\begin{array}{c}
J_{0}(\frac{a}{b})+\overset{\infty }{\underset{k=1}{\dsum }}2(-1)^{k}J_{2k}(%
\frac{a}{b})\cos (2kbw^{2}s) \\ 
+\frac{a}{b}\overset{\infty }{\underset{k=1}{\dsum }}(-1)^{k}J_{2k-1}(\frac{a%
}{b})\left( \cos \left( kbw^{2}s\right) +\cos \left( 2(k-1)bw^{2}s\right)
\right)%
\end{array}%
\right)
\end{equation*}%
When we integrate above equations, we have a curve $J^{3}(\gamma )(s)=\left(
x(s),y(s),z(s)\right) $ where 
\begin{equation*}
\begin{array}{c}
x(s)=-\epsilon \frac{a^{3}w^{2}}{4}J_{0}(\frac{a}{b})\left( 
\begin{array}{c}
\frac{1}{\left( 1+bw\right) (1+2bw)^{2}}\cos \left( w(1+2bw)s\right) +\frac{2%
}{a^{2}w^{2}}\cos \left( ws\right) \\ 
+\frac{1}{\left( 1-bw\right) (1-2bw)^{2}}\cos \left( w(1-2bw)s\right)%
\end{array}%
\right) \\ 
-\epsilon \frac{a^{3}w^{2}}{4}\left( 
\begin{array}{c}
\frac{1}{1+bw}\overset{\infty }{\underset{k=1}{\dsum }}(-1)^{k}J_{2k}(\frac{a%
}{b})\left( 
\begin{array}{c}
\frac{1}{\left( 2bw\left( k+1\right) +1\right) ^{2}}\cos \left( \left(
2bw\left( k+1\right) +1\right) ws\right) \\ 
+\frac{1}{\left( 2bw\left( k-1\right) -1\right) ^{2}}\cos \left( \left(
2bw\left( k-1\right) -1\right) ws\right)%
\end{array}%
\right) \\ 
+\frac{1}{1-bw}\overset{\infty }{\underset{k=1}{\dsum }}(-1)^{k}J_{2k}(\frac{%
a}{b})\left( 
\begin{array}{c}
\frac{1}{\left( 2bw\left( k-1\right) +1\right) ^{2}}\cos \left( \left(
2bw\left( k-1\right) +1\right) ws\right) \\ 
+\frac{1}{\left( 2bw\left( k+1\right) -1\right) ^{2}}\cos \left( \left(
2bw\left( k+1\right) -1\right) ws\right)%
\end{array}%
\right) \\ 
-\frac{1}{a^{2}w^{2}}\overset{\infty }{\underset{k=1}{\dsum }}%
2(-1)^{k}J_{2k}(\frac{a}{b})\left( 
\begin{array}{c}
\frac{1}{\left( 2kbw+1\right) ^{2}}\cos \left( \left( 2kbw+1\right) ws\right)
\\ 
+\frac{1}{\left( 2kbw-1\right) ^{2}}\cos \left( \left( 2kbw-1\right)
ws\right)%
\end{array}%
\right)%
\end{array}%
\right)%
\end{array}%
\end{equation*}%
\begin{equation*}
\begin{array}{c}
y(s)=-\epsilon \frac{a^{3}w^{2}}{4}J_{0}(\frac{a}{b})\left( 
\begin{array}{c}
\frac{1}{\left( 1+bw\right) (1+2bw)^{2}}\sin \left( w(1+2bw)s\right) +\frac{2%
}{a^{2}w^{2}}\sin \left( ws\right) \\ 
+\frac{1}{\left( 1-bw\right) (1-2bw)^{2}}\sin \left( w(1-2bw)s\right)%
\end{array}%
\right) \\ 
-\epsilon \frac{a^{3}w^{2}}{4}\left( 
\begin{array}{c}
\frac{1}{1+bw}\overset{\infty }{\underset{k=1}{\dsum }}(-1)^{k}J_{2k}(\frac{a%
}{b})\left( 
\begin{array}{c}
\frac{1}{\left( 2bw\left( k+1\right) +1\right) }\sin \left( \left( 2bw\left(
k+1\right) +1\right) ws\right) \\ 
+\frac{1}{\left( 2bw\left( k-1\right) -1\right) }\sin \left( \left(
2bw\left( k-1\right) -1\right) ws\right)%
\end{array}%
\right) \\ 
+\frac{1}{1-bw}\overset{\infty }{\underset{k=1}{\dsum }}(-1)^{k}J_{2k}(\frac{%
a}{b})\left( 
\begin{array}{c}
\frac{1}{\left( 2bw\left( k-1\right) +1\right) }\sin \left( \left( 2bw\left(
k-1\right) +1\right) ws\right) \\ 
+\frac{1}{\left( 2bw\left( k+1\right) -1\right) }\sin \left( \left(
2bw\left( k+1\right) -1\right) ws\right)%
\end{array}%
\right) \\ 
+\frac{1}{a^{2}w^{2}}\overset{\infty }{\underset{k=1}{\dsum }}%
2(-1)^{k}J_{2k}(\frac{a}{b})\left( 
\begin{array}{c}
\frac{1}{\left( 2kbw+1\right) }\sin \left( \left( 2kbw+1\right) ws\right) \\ 
+\frac{1}{\left( 2kbw-1\right) }\sin \left( \left( 2kbw-1\right) ws\right)%
\end{array}%
\right)%
\end{array}%
\right)%
\end{array}%
\end{equation*}%
\begin{equation*}
z(s)=-\frac{\epsilon b}{aw}\left( 
\begin{array}{c}
J_{0}(\frac{a}{b})s+\overset{\infty }{\underset{k=1}{\dsum }}\frac{%
(-1)^{k}J_{2k}(\frac{a}{b})}{kbw^{2}}\sin (2kbw^{2}s) \\ 
+\frac{a}{b^{2}w^{2}}\overset{\infty }{\underset{k=1}{\dsum }}%
(-1)^{k}J_{2k-1}(\frac{a}{b})\left( \frac{1}{k}\sin \left( kbw^{2}s\right) +%
\frac{1}{2(k-1)}\sin \left( 2(k-1)bw^{2}s\right) \right)%
\end{array}%
\right)
\end{equation*}%
So we have a 2-slant curve $J^{3}(\gamma )(s)$.

\begin{definition}
In 3-Euclidean spaces, if $\ \gamma =S^{1}(\overrightarrow{a},r)$ , then we
have a set $%
\mathbb{Z}
(\gamma )=\left\{ ...,J^{-2}(\gamma ),J^{-1}(\gamma ),\gamma ,J(\gamma
),J^{2}(\gamma ),...\right\} $. We said that this set is constant precession
curve chain. So, the set of all the $N_{k}$-constant procession curve $(%
\mathbb{Z}
^{N})$ is given by 
\begin{equation*}
\mathbb{Z}
^{N}=\underset{\overrightarrow{a}\in E^{3},r\rangle 0}{\cup }%
\mathbb{Z}
\left( S^{1}(\overrightarrow{a},r)\right)
\end{equation*}%
Furhermore, the set of all $k$-slant curve in Euclidean $3-$spaces is given
by $%
\mathbb{Z}
^{J}\cup S^{I}$ where%
\begin{equation*}
\mathbb{Z}
^{J}=\underset{\gamma \text{ planer curve}}{\cup }%
\mathbb{Z}
\left( \gamma \right)
\end{equation*}
\end{definition}

\section{\textbf{k-slant curve and magnetic curve}}

Let $M$ be $2n+1-$smooth manifold. If\ there exists $(\phi ,\xi ,\eta ,g)$
structure such that, for all $X,Y\in \chi (M)$,%
\begin{equation*}
\begin{array}{c}
\phi ^{2}(X)=-X+\ \eta (X)\xi ,g(\phi X,\phi Y)=g(X,Y)-\eta (X)\eta (Y), \\ 
\eta (X)=g(X,\xi ),\eta (\xi )=1,\phi (\xi )=0,\eta o\phi =0%
\end{array}%
\text{ }
\end{equation*}%
then it is said that $(M,\phi ,\xi ,\eta ,g)$ is almost contact metric
manifold where $\phi ,\xi ,\eta $ are $(1,0),(1,0),(0,1)$ (resp.) type
tensor and $g$ is a metric (\cite{BLA}). In almost contact metric manifold,
it is said that $\Phi $ is fundamental form of almost contact metric
manifold where $\Phi (X,Y)=g(\phi X,Y)$ (\cite{BLA}). Let $(M,\phi ,\xi
,\eta ,g)$ be $3-$dimensional almost contact metric manifold. Extended of
the cross product are defined by 
\begin{equation*}
X\wedge Y=-g(\phi X,Y)-\eta (Y)\phi X+\eta (X)\phi Y
\end{equation*}%
for all $X,Y\in \chi (M)$ (\cite{CC}). So we have\ 
\begin{equation*}
\phi (X)=\xi \wedge X.
\end{equation*}%
In 3-Euclidean spaces, if we define a set 
\begin{equation*}
V=\left\{ (v_{1},v_{2},0):v_{1},v_{2}\in R\right\}
\end{equation*}%
then $V$ is subspace of $R^{3}$. So we can define natural projection from $%
R^{3}$ to $V$ by $\pi (v_{1},v_{2},v_{3})=(v_{1},v_{2},0)$ and almost
complex map $J(v_{1},v_{2},0)=(-v_{2},v_{1},0)$ on $V$. If we define $\phi
=J\circ \pi $, $\eta =dz$, $\xi =\frac{\partial }{\partial z}$, then $\left(
R^{3},\phi ,\xi ,\eta ,g\right) $ is almost contact metric manifold where $g$
is standart Euclid metric (\cite{CC}). In this case, $X\wedge Y=X\times Y$
where $\times $ is the known as cross product (\cite{CC}). Let $M$ be a
unit-speed regular curve with coordinate neighborhood $(I,\gamma )$ and $%
\left\{ T,N,B,\kappa ,\ \tau \right\} $ be Serret-Frenet apparatus of the
curves. So we have $T\wedge N=B,N\wedge B=T,B\wedge T=N.$ Following
equations are hold%
\begin{eqnarray*}
i)\text{ }\phi (T) &=&\eta (B)N-\eta (N)B \\
ii)\text{ }\phi (N) &=&\eta (T)B-\eta (B)T \\
iii)\text{ }\phi (B) &=&\eta (N)T-\eta (T)N
\end{eqnarray*}%
where $\xi =\eta (T)T+\eta (N)N+\eta (B)B$ and $\eta (T)^{2}+\eta
(N)^{2}+\eta (B)^{2}=1$(\cite{CC}). Let $\xi $ be a magnetic field and $\Phi 
$ be a close 2-form on $M^{3}$ where $\Phi (X,Y)=g(\phi X,Y)$ and $\phi
(X)=\xi \wedge X$. $\phi $ is the Lorentz force of $\Phi $. Let $\left(
\gamma \right) $ be regular curve on $M^{3}$. If 
\begin{equation*}
\nabla _{T}T=\phi (T)=\xi \times T\text{ \ (Landau-Hall equation)}
\end{equation*}%
then $\left( \gamma \right) $ is the magnetic curve of $\left( M,g,\Phi
\right) $ where $\nabla $ is Levi-Civita connection of $g$ (\cite{BARR}). In
this case, $\left( M^{3},\phi ,\xi ,\eta ,g\right) $ is almost contact
metric manifold and $\Phi $ is the fundamental form of this manifold. \ From
Landau-Hall equation, we have%
\begin{equation*}
\left( 
\begin{array}{c}
\phi (T) \\ 
\phi (N) \\ 
\phi (B)%
\end{array}%
\right) =\left( 
\begin{array}{ccc}
0 & \kappa & 0 \\ 
-\kappa & 0 & w \\ 
0 & -w & 0%
\end{array}%
\right) \left( 
\begin{array}{c}
T \\ 
N \\ 
B%
\end{array}%
\right) .
\end{equation*}%
$\left( \gamma \right) $ is magnetic curve if and only if $\xi =wT+\kappa B$
(\cite{CAB}). \ Bozkurt (at al. (\cite{BI}) ) defined a new type Landau-Hall
equation such that 
\begin{equation*}
\nabla _{T}N=\phi (N)=\xi \times N
\end{equation*}%
where $N$ is a normal vector field along the curve. If 
\begin{equation*}
\nabla _{T}N=\phi (N)=\xi \times N
\end{equation*}%
then $\left( \gamma \right) $ is the $N-$magnetic curve of $\left( M,g,\Phi
\right) $(\cite{BI}). From Landau-Hall equation, we have 
\begin{equation*}
\left( 
\begin{array}{c}
\phi (T) \\ 
\phi (N) \\ 
\phi (B)%
\end{array}%
\right) =\left( 
\begin{array}{ccc}
0 & \kappa & \Omega \\ 
-\kappa & 0 & \tau \\ 
-\Omega & -\tau & 0%
\end{array}%
\right) \left( 
\begin{array}{c}
T \\ 
N \\ 
B%
\end{array}%
\right)
\end{equation*}%
(\cite{BI}). $\left( \gamma \right) $ is the $N-$magnetic curve if and only
if $\xi =\tau T+\kappa B-\Omega N$ (\cite{BI}). Therefore, Bozkurt (at al. (%
\cite{BI}) ) defined a $B-$magnetic curve of $\left( M,g,\Phi \right) $(\cite%
{BI}). If 
\begin{equation*}
\nabla _{T}B=\phi (B)=\xi \times B
\end{equation*}%
then $\left( \gamma \right) $ is the $B-$magnetic curve of $\left( M,g,\Phi
\right) $(\cite{BI}). From Landau-Hall equation, we have 
\begin{equation*}
\left( 
\begin{array}{c}
\phi (T) \\ 
\phi (N) \\ 
\phi (B)%
\end{array}%
\right) =\left( 
\begin{array}{ccc}
0 & w & 0 \\ 
-w & 0 & \tau \\ 
0 & -\tau & 0%
\end{array}%
\right) \left( 
\begin{array}{c}
T \\ 
N \\ 
B%
\end{array}%
\right) .
\end{equation*}%
(\cite{BI}). $\left( \gamma \right) $ is $B-$magnetic curve if and only if $%
\xi =\tau T+wB$(\cite{BI}). Similarly, we can define generalized Landau-Hall
equation such that 
\begin{equation*}
\nabla _{T}Z=\phi (Z)=\xi \times Z\text{ \ (Generalized Landau-Hall equation)%
}
\end{equation*}%
then it is said that $\left( \gamma \right) $ is the $Z-$magnetic curve of $%
\left( M,g,\Phi \right) $. From Generalized Landau-Hall equation, we have
following theorem.

\begin{theorem}
\bigskip $\left( \gamma \right) $ is the $Z-$magnetic curve of $\left(
M,g,\Phi \right) $ if and only if 
\begin{equation*}
\left( 
\begin{array}{c}
Z_{1}^{\prime } \\ 
Z_{2}^{\prime } \\ 
Z_{3}^{\prime }%
\end{array}%
\right) =\left( 
\begin{array}{ccc}
0 & \kappa -\xi _{3} & \xi _{2} \\ 
-\left( \kappa -\xi _{3}\right) & 0 & \tau -\xi _{1} \\ 
-\xi _{2} & -\left( \tau -\xi _{1}\right) & 0%
\end{array}%
\right) \left( 
\begin{array}{c}
Z_{1} \\ 
Z_{2} \\ 
Z_{3}%
\end{array}%
\right)
\end{equation*}%
where $\xi =\xi _{1}T+\xi _{2}N+\xi _{3}B$ and $Z=Z_{1}T+Z_{2}N+Z_{3}B$
\end{theorem}

If 
\begin{equation*}
\nabla _{T}N_{k}=\phi (N_{k})=\xi \times N_{k}\text{, }
\end{equation*}%
then it is said that \bigskip $\left( \gamma \right) $ is the $N_{k}$%
-magnetic curve. \ Lorentz force in the Serret-Frenet frame of $M_{k}$ is
given by 
\begin{equation*}
\left( 
\begin{array}{c}
\phi (T_{k}) \\ 
\phi (N_{k}) \\ 
\phi (B_{k})%
\end{array}%
\right) =\left( 
\begin{array}{ccc}
0 & \kappa _{k} & \Omega _{k+1} \\ 
-\kappa _{k} & 0 & \tau \\ 
-\Omega _{k+1} & -\tau _{k} & 0%
\end{array}%
\right) \left( 
\begin{array}{c}
T_{k} \\ 
N_{k} \\ 
B_{k}%
\end{array}%
\right)
\end{equation*}%
and Serret-Frenet formula of the curves is given by%
\begin{equation*}
\left( 
\begin{array}{c}
T_{k}%
{\acute{}}
\\ 
N_{k}^{%
{\acute{}}%
} \\ 
B_{k}^{%
{\acute{}}%
}%
\end{array}%
\right) =\left( 
\begin{array}{ccc}
0 & \kappa _{k} & 0 \\ 
-\kappa _{k} & 0 & \tau _{k} \\ 
0 & -\tau _{k} & 0%
\end{array}%
\right) \left( 
\begin{array}{c}
T_{k} \\ 
N_{k} \\ 
B_{k}%
\end{array}%
\right) .
\end{equation*}%
Thus, \bigskip $\left( \gamma \right) $ is $N_{k}$-magnetic curve\ of the
magnetic field $\xi $ if and only if 
\begin{equation*}
\xi =\tau _{k}T_{k}-\Omega _{k+1}N_{k}+\kappa _{k}B_{k}\in Ker\phi .
\end{equation*}%
If we derive $\xi $ along the curve, \ then we obtain 
\begin{equation*}
\nabla _{T}\xi =\left( \tau _{k}^{\prime }+\Omega _{k+1}\kappa _{k}\right)
T_{k}-\Omega _{k+1}^{\prime }N_{k}+\left( \kappa _{k}^{\prime }-\Omega
_{k+1}\tau _{k}\right) B_{k}.
\end{equation*}%
If $\xi $ is constant vector field along the curve, we have 
\begin{equation}
\tau _{k}^{\prime }=-\Omega _{k+1}\kappa _{k}  \label{aa}
\end{equation}%
and 
\begin{equation}
\kappa _{k}^{\prime }=\Omega _{k+1}\tau _{k}  \label{b}
\end{equation}%
where $\Omega _{k+1}=const$. From (\ref{aa}) and (\ref{b}), \ we get 
\begin{equation}
\kappa _{k}=R\cos \left( \Omega _{k+1}s+c_{0}\right)  \label{c}
\end{equation}%
\begin{equation}
\tau _{k}=R\sin \left( \Omega _{k+1}s+c_{0}\right)  \label{d}
\end{equation}%
where $R=const$. \ From (\ref{c}) and (\ref{d}), we have $M_{k}=J^{2}(S^{1}(%
\overrightarrow{a},r))$ and $M=J^{k+2}(S^{1}(\overrightarrow{a},r))$.
Furthermore, Ramiz (at al.) defined $N_{k}$-constant procession curve in $3-$%
Euclidean spaces (\cite{CY}). In this spaces, Darboux vector of $M_{k}$ is
defined as 
\begin{equation*}
W_{k}=\tau _{k}T_{k}+\kappa _{k}B_{k}
\end{equation*}%
and 
\begin{equation*}
A_{k}=W_{k}\pm \Omega _{k+1}N_{k}
\end{equation*}%
where $\Omega _{k+1}=const$ (\cite{CY}). Then, $M$ is said $N_{k}$-constant
procession curve in Euclidean $3-$spaces if there is constant angle between $%
W_{k}$ and fixed direction $A_{k}$ (\cite{CY}). From Theorem 4 in(\cite{CY}%
), following equations are equivalent.

$i)$ $M$ is the $N_{k}$-constant procession curve. $\ $

$ii)$ $\kappa _{k}=R\cos \left( \Omega _{k+1}s+c_{0}\right) $, $\tau
_{k}=R\sin \left( \Omega _{k+1}s+c_{0}\right) $.

where $\Omega _{k+1}$ and $c_{0}$ are constant. So we give following theorem

\begin{theorem}
In $3-$Euclidean spaces, a $N_{k}$-magnetic curve \ is a $N_{k}$-constant
procession curve if and only if \ $\xi =\tau _{k}T_{k}-\Omega
_{k+1}N_{k}+\kappa _{k}B_{k}$ is constant vector field along the curve.
\end{theorem}


\begin{thebibliography}{99}
\bibitem{Ali} Ali TA. New special curves and their spherical indicatrices,
arXive: 0909.2390v[Math. DG]13 sep.2009.

\bibitem{BARR} Barros M, Romero A, Cabrerizo JL, Fern\'{a}ndez M. The
Gauss-Landau-Hall problem on Riemannian surfaces, J. Math. Phys. 46 (2005),
112905:1--15.

\bibitem{BLA} Blair DE. Contact manifolds in Riemannian geometry, Lecture
Notes in Math. 509, Springer, Berlin, Hiedelberg, New York, (1976).

\bibitem{BL} Blaschke W. Bemerkungen \"{u}ber allgemeine Schraubenlinien.
Monatsh. f\"{u}r.Math. und Phys. 19, 188-204, 1908.

\bibitem{RB} Blum R. A remarkable class of Mannheim curves, Canad. Math.
Bull. 9, 223-228 (1966).

\bibitem{BI} Bozkurt Z, G\"{o}k I, Yayl\i\ Y, Ekmekci FN. A new approach for
magnetic curves in Riemannian 3D-manifolds. J Math Phys 2014; 55: 1--12.

\bibitem{BG} Breuer S., Gottlieb D. Explicit characterization of spherical
curves, Proc. Math. Soc. 27(1971), 126-127.

\bibitem{CAB} Cabrerizo JL. Magnetic fields in 2D and 3D sphere, J.
Nonlinear Math. Phys., 20 (2013), no.3,440--450.

\bibitem{CC} Camc\i\ C. Extended cross product in a 3-dimensionsal almost
contact metric manifold with applications to curve theory, Turk. J. Math.,
35(2011), 1-14.

\bibitem{C} Camc\i\ C, Kula L, Alt\i nok M. On spherical slant helices in
Euclidean 3- space, arXiv:1308.5532v3 [math.DG] 25 Sep 2013.

\bibitem{F} Frenet F. Sur les courbe \`{a} double courbure, Jour. de Math.
17, 437-447 (1852) Extrait d'une These (Toulouse, 1847).

\bibitem{ITB} Izumiya S, Tkeuchi N. Generic properties of helices and
Bertrand curves, J. geom., 74, 97-109, (2002).

\bibitem{IT} Izumiya S, Tkeuchi N. New special curves and developable
surfaces, Turk J. Math., 28, 153-163, (2004).

\bibitem{EK} Kreyszig E. Diferential Geometry, Mathematical Exposition,
11.Uni.of Toronto Press (1959).

\bibitem{KY} Kula L, Yayli Y. On slant helix and its spherical indicatrix.
Applied Mathematics and computation, 169, 600-607, (2005)

\bibitem{M} Mannheim A. Paris C.R. 86(1978) p.1254-1256.

\bibitem{MEN} Menninger Anton, Characterization of the slant helix as
successor curve of the general helix, International Electronic of Geometry,
Volume 7, No.2, 84-91(2014)

\bibitem{L} Lancret MA. M\'{e}moire sur les courbes \`{a} double courbure, M%
\'{e}moires pr\'{e}sent\'{e}s \`{a} l'Institut des sciences, lettres et arts
par divers savants, 1802, tome 1, (1806), p. 416-454.

\bibitem{CY} Ramis C, Uzunoglu B, Yayli Y. New associated curve k-principle
direction curves and $N_{k}$-slant helix, arXive: 1404.7369[Math. DG]28
Apr.2014.

\bibitem{SAL} Salkowski, E.: Zur Transformation von Raumkurven,
Mathematische Annalen, 66(4), 517-557(1909).

\bibitem{SV} Saint-Venant \ JC. M\'{e}moire sur les lignes courbes non
planes, Journal de l'\'{E}cole polytechnique, cahier (1845), p. 1-76.

\bibitem{S} Serret JA. Sur quelques formules relatives \`{a} la th\'{e}orie
des courbes \`{a} double courbure, Journal de math\'{e}matiques pures et
appliqu\'{e}es, 16 (1851), p. 193-207.

\bibitem{Sco} Scofield PD. Curves of constant precession, Amer. Math.
Monthly 102, 531--537 (1995).

\bibitem{ST} Struik DJ. Lectures on Classical Differential Geometry, Dover,
New-York, (1988).

\bibitem{TAK} Takahashi T., Tkeuchi N. , Clad helices and developable
surfaces, Bulletin of Tokyo University, 66, 1-9(2014).

\bibitem{UZ} Uzunoglu B, G\"{o}k I. and Yayl\i\ Y., A new approach on curves
of constant precession, Applied Mathematics and Computation, 275 (2016)
317--323.

\bibitem{wong} Wong YC. On an explicit characterization of spherical curves.
Proc. Amer. Math. Soc., 34(1) 239--242, 1972.

\bibitem{WWB} Bell WW. Special Functions for Scientists and Engineers, D.
van Nostrand Comp.Ltd., (1968).
\end{thebibliography}
\end{document}